\documentclass{amsart}

\theoremstyle{theorem}
\newtheorem{theorem}{Theorem}

\theoremstyle{definition}

\newtheorem*{remark}{Remark}

\newtheorem{proposition}[theorem]{Proposition}

\newtheorem{example}[theorem]{Example}
\newtheorem{xca}{Exercise}

\newcommand{\N}{\mathbb{N}}

\newcommand{\R}{\mathbb{R}}

\newcommand{\calL}{\mathcal{L}}

\begin{document}

\title{A Lecture on Integration by Parts}
\author{}
 
\maketitle



\centerline{\scshape John A.~Rock}
\medskip
{\footnotesize
 \centerline{Cal Poly Pomona}
   \centerline{jarock@cpp.edu}
}

\begin{abstract}
Integration by parts (IBP) has acquired a bad reputation. While it allows us to compute a wide variety of integrals when other methods fall short, its implementation is often seen as plodding and confusing. Readers familiar with \textit{tabular} IBP understand that, in particular cases, it has the capacity to significantly streamline and shorten computation. In this paper, we establish a tabular approach to IBP that is completely general and, more importantly, a powerful tool that promotes exploration and creativity. 
\end{abstract}

\vspace{5mm}

\noindent
Integration by parts (IBP) can always be utilized with a tabular approach. More importantly, tabular IBP is a powerful tool that promotes exploration and creativity. If mentioned at all, the explanation of a tabular method for IBP in a calculus textbook is often perfunctory or relegated to the section of exercises. (For example, see \cite[\S 7.2]{BriggsCochGill}.) The approach developed here allows us to readily compute integrals using IBP, quickly learn from poor choices when setting up IBP, and neatly derive formulas corresponding to important results such as Taylor's Formula with Integral Remainder (see Theorem \ref{thm:Taylors} as well as \cite{Horo}.) 

Our technique actually has a straightforward justification. Recall the product rule for derivatives: For suitable functions $u$ and $v$ of a real variable $x$, we have
\begin{center}
	$(uv)'=uv'+u'v.$
\end{center}
Integrating both sides of this equation, applying the definition of antiderivative, and changing variables accordingly yields 
\begin{center}
	$uv=\int u\,dv+\int v\,du.$
\end{center}
A slight rearrangement produces the familiar formula (``ultra-violet voodoo'')
\begin{center}
	$\int u\,dv=uv-\int v\,du.$
\end{center}

In general, the idea behind IBP is to let $\int u\,dv$ denote a given integral in the hopes that a new integral $\int v\,du$ can be readily determined or managed. The pair $u$ and $dv$ are chosen so that their product represents the given integrand; and the pair $v$ and $du$ are determined by $\int\,dv=\int v_0\,dx$ (usually excluding a constant) and $u'=du/dx$, respectively.\footnote{The acronym LIPET (which stands for ``logarithm, inverse trigonometric, polynomial, exponential, then trigonometric'') helps us choose an appropriate $u$ in cases where the given integrand is a product of relatively simple functions, and other techniques do not suffice.} It is important to note that, like substitution, \textit{with IBP we replace one integral with another}, always leaving an integral that must be resolved. The tabular method described here works the same way.

To motivate the use of a tabular approach, suppose we want to iterate IBP as follows (this will not always be the case):  Given $\int u\,dv$, let $u_1:=u$, $v_0\,dx:= dv$, and $v_1:=\int\,dv=\int v_0\,dx$. For each integer $j\geq 2$, let $u_j:=u_{j-1}'$ and $v_j:=\int v_{j-1}\,dx$. (We assume throughout that the functions involved behave well enough.) Then, for each integer $n\geq 2$, we have
\begin{align} \textstyle
\int u\,dv
	&\textstyle = uv-\int v\,du	\notag\\
	&\textstyle = u_1v_1 - u_2v_2 +\int v_2\,du_2 \notag\\
	&\textstyle = u_1v_1 - u_2v_2 + u_3v_3 - \int v_3\,du_3 \notag\\
	\cdots ~ &\textstyle = u_1v_1-u_2v_2+u_3v_3-\cdots + (-1)^{n-1}u_nv_n+ (-1)^{n}\int v_n\,du_n. \label{eqn:Initial}
\end{align}
For a full proof of \eqref{eqn:Initial}, see \cite{Folley}. Note that in  \eqref{eqn:Initial} we suppress the addition of a constant. Throughout the paper, $C$ and $C_0$ denote constants.

Assuming we iterate IBP as in \eqref{eqn:Initial}, when should we stop?  The answer depends on how we want to handle the integrals $(-1)^{j}\int v_j\,du_j$, \textit{as each one is generated}. The simplest case occurs when some $(-1)^{j}\int v_j\,du_j$ is readily computed, but the more difficult cases prove to be the most interesting. Also, if we are not sure how to proceed, we can stop at any point and assess the situation. 

To create a table that leads to \eqref{eqn:Initial}, label three columns $+/-$, $u$, and $dv$. Start the first row with a $+$ sign in the $+/-$ column, a chosen $u_1$ in the $u$ column, and $v_0 $ (where $dv = v_0\,dx$ and the differential $dx$ is suppressed) in the $dv$ column. With each new row, alternate (alt.)~between $+$ and $-$ in  $+/-$ column, differentiate (diff.)~the function in the $u$ column, and integrate (int.)~the function in the $dv$ column. Also with each row, decide how to proceed based on the integral of the product of the terms in that row.
\begin{center}
	\begin{tabular}{c|cccl} 
	(alt.)	&(diff.)&	& (int.)	&\\
	$+/-$ 	& $u$  	& 	& $dv$	&\\ \hline \\
	$+$		& $u_1$ 	& 	& $v_0$	& (The top row contains the original integrand.)\\
			&		& $\searrow$	&	& \\	
	$-$		& $u_2$	&	& $v_1$ &  (Is $-\int u_2v_1 = -\int v_1\,du_1$ manageable? \\
			&		& $\searrow$	& 	& \hspace{1mm} If so, stop. If not, continue.)\\	 
	$+$		& $u_3$	& 	& $v_2$ & \\ 
	$\vdots$		&	$\vdots$	& & $\vdots$	& \\	 
	$(-1)^{n-1}$		& $u_n$	& 	& $v_{n-1}$ &\\ 
			&		& $\searrow$	& &\\
	$(-1)^{n}$	& $u_{n+1}$		&$\rightarrow$& $v_n$ & (The bottom row represents $(-1)^{n}\int v_n\,du_n$.)
	\end{tabular}
\end{center}
Each of the products $(-1)^{j-1}u_jv_j$ in \eqref{eqn:Initial} are obtained by taking the product of $(-1)^{j-1}$ from $jth$ row of the $+/-$ column, $u_j$ from the $j$th row of the $u$ column, and $v_j$ from \textit{one row further down} in the $(j+1)$th row of the $dv$ column (i.e., $u_j$ and $v_j$ match up diagonally, as indicated by $\searrow$). Summing together yields \eqref{eqn:Initial}:
\begin{align*}
	\textstyle \int u\,dv	&= \textstyle u_1v_1-u_2v_2+u_3v_3+\cdots + (-1)^{n}\int v_n\,du_n.
\end{align*}

It is important to note that we can stop at any point, and we decide when to stop based on the information provided in the last row (which represents the integral $(-1)^{n}\int v_n\,du_n$). Also, we may not always want to iterate IBP using $u_j:=u_{j-1}'$ and $v_j:=\int v_{j-1}\,dx$ for $j\geq 2$ as above. For instance, it may be preferable to simplify the integrand of some $(-1)^{j}\int v_j\,du_j$ and apply tabular IBP separately to the simplified integral. See Example \ref{eg:Twice}.

In the remainder of the paper we discuss several examples from calculus and analysis. Further examples can be found in \cite{Horo} and \cite{Murty}, as well as the references therein. Readers are \textit{strongly} encouraged to work out the examples themselves by hand using tabular IBP! Remember to construct the tables \textit{one row at a time.}

\begin{example}
\label{eg:lnx}
Consider $\int \ln{x}\,dx$ where $x>0$. Let $u=\ln{x}$ and $dv=1\,dx$. Note that after generating the second row of the table, we consider the integral of the product of the functions in this row in order to decide what to do next.

\begin{center}
	\begin{tabular}{c|cccl} 
	(alt.)	&(diff.)&	& (int.)	&\\
	$+/-$ 	& $u$  	& 	& $dv$	&\\ \hline \\
	$+$		& $\ln{x}$ 	& 	& $1$	& \\
			&		& $\searrow$	&	& \\	
	$-$		& $\displaystyle \frac{1}{x}$	&$\rightarrow$	& $x$ & ($-\int \frac{1}{x}\cdot x \,dx= -\int \,dx =-x+C$) \\
	\end{tabular}
\end{center}
By \eqref{eqn:Initial}, we have 
\[
\int \ln{x} \,dx= x\ln{x}-x+C.
\]
\end{example}

\begin{example}
Consider $\int e^{3x}\sin{2x}\,dx$. Let $u = e^{3x}$ and $dv = \sin{2x}\,dx$. 
\begin{center}
	\begin{tabular}{c|cccl} 
	(alt.)		&(diff.)&	& (int.)	&\\
	$+/-$ 	& $u$ &	& $dv$ & \\ \hline \\
	$+$		& $e^{3x}$ 	& 	& $\sin{2x}$ &  \hspace{5mm}  (The original integrand.)\\
			&		& $\searrow$	& &\\	
	$-$		& $3e^{3x}$	&	& $\displaystyle -\frac{1}{2}\cos{2x}$ & \hspace{5mm} ($  +\frac{3}{2}\int e^{3x}\cos{2x} \,dx$, try another step.)\\
			&		& $\searrow$	& &\\	 
	$+$		& $9e^{3x}$	& $\rightarrow$	& $\displaystyle -\frac{1}{4}\sin{2x}$&  \hspace{5mm} ($-\frac{9}{4}\int e^{3x}\sin{2x} \,dx$, a copy of\\
	& & & & \hspace{8mm} the original integral.)  
	\end{tabular}
\end{center}
At this point, we see that the last row generates a copy of the original integral, so  we stop and see where things stand. We have
\begin{align*}
	\int e^{3x}\sin{2x}\,dx	&= -\frac{e^{3x}}{2}\cos{2x} +\frac{3e^{3x}}{4}\sin{2x} -\frac{9}{4}\int e^{3x}\sin{2x} \,dx +C.
\end{align*}
Adding $\frac{9}{4}\int e^{3x}\sin{2x} \,dx$ to both sides then multiplying both sides by $\frac{4}{13}$ yields 
\begin{align*}
	\int e^{3x}\sin{2x} \,dx	&=  \frac{e^{3x}}{13}\left(3\sin{2x}-2\cos{2x}\right)+C_0.
\end{align*}
\end{example}

\begin{example}
Consider $\int (x^2-3x)\sin{x}\,dx$. For the sake of exploration and discovery (in this case, learning from a mistake), let $u=\sin{x}$ and $dv=(x^2-3x)\,dx$. This will quickly prove to be a bad choice.
\begin{center}
	\begin{tabular}{c|cccl} 
	(alt.)	&(diff.)&	& (int.)	&\\
	$+/-$ 	& $u$  	& 	& $dv$	&\\ \hline \\
	$+$		& $\sin{x}$ 	& 	& $x^2-3x$&\\
			&		& $\searrow$	&	& \\	
	$-$		& $\cos{x}$	&$\rightarrow$ & $\displaystyle \frac{x^3}{3}-\frac{3x^2}{2}$ &  $\displaystyle \left(\textnormal{Consider } \int \left(\frac{3x^2}{2}-\frac{x^3}{3}\right)\cos{x}\,dx \textnormal{.}\right)$
	\end{tabular}
\end{center}

We can immediately see that the integral $\displaystyle \int \left(\frac{3x^2}{2}-\frac{x^3}{3}\right)\cos{x}\,dx$ is at least as difficult as the original $\displaystyle \int (x^2-3x)\sin{x}\,dx$.  Nevertheless, we can apply \eqref{eqn:Initial} to get 
\begin{align*}
\int (x^2-3x)\sin{x}\,dx	&=  \left(\frac{x^3}{3}-\frac{3x^2}{2}\right)\sin{x}+ \int \left(\frac{3x^2}{2}-\frac{x^3}{3}\right)\cos{x}\,dx +C. 
\end{align*}
While this is technically true, it is certainly not what we want. 

Try again. This time, let $u=(x^2-3x)$ and $dv=\sin{x}\,dx$, and consider the options with the integral generated by each new row, \textit{one row at a time}.
\begin{center}
	\begin{tabular}{c|cccl} 
	(alt.)	&(diff.)&	& (int.)	&\\
	$+/-$ 	& $u$  	& 	& $dv$	&\\ \hline \\
	$+$		& $x^2-3x$ 	& 	& $\sin{x}$& $\hspace{5mm}$ \\
			&		& $\searrow$	&	& \\	
	$-$		& $2x-3$	&	& $-\cos{x}$ &  $\hspace{5mm}$ ($ -\int (2x-3)\cos{x}\,dx$ simplifies with\\ 
			&		& $\searrow$	& 	& \hspace{10mm} another iteration.)\\
	$+$		& $2$	&$\rightarrow$ 	& $-\sin{x}$ &  $\hspace{5mm}$ ($ -\int 2\sin{x}\,dx=2\cos{x} +C$, \\
	& & & &$\hspace{10mm}$ so stop here.)
	\end{tabular}
\end{center}
By \eqref{eqn:Initial}, we have 
\begin{align}
	\int (x^2-3x)\sin{x}\,dx	&= (3x-x^2)\cos{x}+(2x-3)\sin{x} +2\cos{x} +C. \label{eqn:xsquaredsin}
\end{align}
\end{example}

\begin{remark}
Readers familiar with other versions of tabular IBP may note that we could have easily generated one more row in the previous table, yielding: \vspace{2mm}
\begin{center}
	\begin{tabular}{c|cccl} 
	(alt.)	&(diff.)&	& (int.)	&\\
	$+/-$ 	& $u$  	& 	& $dv$	&\\ \hline \\
	$+$		& $x^2-3x$ 	& 	& $\sin{x}$& $\hspace{5mm}$\\
			&		& $\searrow$	&	& \\	
	$-$		& $2x-3$	&	& $-\cos{x}$ &  $\hspace{5mm}$ \\
			&		& $\searrow$	& 	&\\	 
	$+$		& $2$	& 	& $-\sin{x}$ & \\
		&		& $\searrow$	& 	& \\
	$-$		& $0$	& 	$\rightarrow$ & $\cos{x}$ &  $\hspace{5mm}$ ($ -\int 0\,dx=C$.)
	\end{tabular} \vspace{2mm}
\end{center}
Formula \eqref{eqn:xsquaredsin} immediately follows.

Similarly, for a polynomial $P(x)$ and constants $a\neq 0$, $b\neq 0$, $q\neq 0$, and $q\neq 1$, integrals of the form
\[
\int \frac{P(x)}{(ax+b)^q}\,dx
\]
are readily computed using tabular IBP since successive derivatives of $P(x)$ eventually vanish (see \cite{Horo}). 

It is especially convenient when 0 appears in the $u$ column, which happens if $u$ is chosen to be a polynomial and we iterate as above by setting $u_j:=u_{j-1}'$ and $v_j:=\int v_{j-1}\,dx$ for $j\geq 2$. However, it is not necessary for these conditions to hold in order to know when to stop or for tabular IBP to be effective. Rather, with tabular IBP we proceed by considering our options with the integral $(-1)^{j}\int v_j\,du_j$ generated by each new row, \textit{one row at a time}.
\end{remark}

Our tabular method allows us to compute integrals of the form $\int \sin{ax}\cos{bx}\,dx$, $\int \sin{ax}\sin{bx}\,dx$, and $\int \cos{ax}\cos{bx}\,dx$ where $a\neq b, a\neq 0$, and $b\neq 0$ without resorting to the use trigonometric identities. See \cite{Murty}.

\begin{example}
Consider $\displaystyle \int \sin{2x}\cos{5x}\,dx$. Let $u = \sin{2x}$ and $dv = \cos{5x}\,dx$. 
\begin{center}
	\begin{tabular}{c|cccl} 
	(alt.)		&(diff.)&	& (int.)	&\\
	$+/-$ 	& $u$ &	& $dv$ & \\ \hline \\
	$+$		& $\sin{2x}$ 	& 	& $\cos{5x}$ & \\
			&		& $\searrow$	& &\\	
	$-$		& $2\cos{2x}$	&	& $\displaystyle \frac{1}{5}\sin{5x}$ & \\
			&		& $\searrow$	& &\\	 
	$+$		& $-4\sin{2x}$	& $\rightarrow$	& $\displaystyle -\frac{1}{25}\cos{5x}$&  \hspace{5mm} $ (+\frac{4}{25}\int \sin{2x}\cos{5x} \,dx$, \\& & & & \hspace{7mm} a copy of the original integral.)
	\end{tabular}
\end{center}
By \eqref{eqn:Initial} we have
\begin{align*}
	\int \sin{2x}\cos{5x}\,dx	=~& \frac{1}{5}\sin{2x}\sin{5x} +\frac{2}{25}\cos{2x}\cos{5x}\\ &+\frac{4}{25}\int \sin{2x}\cos{5x} \,dx +C.
\end{align*}
Hence, $\displaystyle \int \sin{2x}\cos{5x}\,dx = \frac{5}{21}\sin{2x}\sin{5x} +\frac{2}{21}\cos{2x}\cos{5x}+C_0.$
\end{example}

\begin{remark}
Integrals like $\int \sin^2{ax}\,dx$ for nonzero $a$ also follow from IBP after an application of the Pythagorean identity $\sin^2{ax}=1-\cos^2{ax}$, as shown in \cite{Murty}.
\end{remark}

\begin{example}
\label{eg:Twice}
In this example, tabular IBP is applied twice, but not in a single table. Consider $\int (3x^2-x)\ln^2{x}\,dx$ where $x>0$. Let $u=\ln^2{x}$ and $v=(3x^2-x)\,dx$.
\begin{center}
	\begin{tabular}{c|cccl} 
	(alt.)	&(diff.)&	& (int.)	&\\
	$+/-$ 	& $u$  	& 	& $dv$	&\\ \hline \\
	$+$		& $\ln^2{x}$ 	& 	& $3x^2-x$	& \\
			&		& $\searrow$	&	& \\	
	$-$		& $\displaystyle \frac{2}{x}\ln{x}$	&  $\rightarrow$ &$\displaystyle x^3-\frac{x^2}{2}$ &  \hspace{5mm} 
	$\displaystyle \left(-\int (2x^2-x)\ln{x}\,dx \right)$ \\
	\end{tabular}
\end{center}
If it is not clear what to do next, stop and assess the situation. So far we have
\begin{align}
\label{eqn:SoFar}
	\int (3x^2-x)\ln^2{x}\,dx&= \left(x^3-\frac{x^2}{2}\right)\ln^2{x} -\int (2x^2-x)\ln{x}\,dx +C.
\end{align}
To evaluate $\int (2x^2-x)\ln{x}\,dx$, we apply tabular IBP again but in a separate table.
\begin{center}
	\begin{tabular}{c|cccl} 
	(alt.)	&(diff.)&	& (int.)	&\\
	$+/-$ 	& $u$  	& 	& $dv$	&\\ \hline \\
	$+$		& $\ln{x}$ 	& 	& $2x^2-x$	& \\
			&		& $\searrow$	&	& \\	
	$-$		& $\displaystyle \frac{1}{x}$	&  $\rightarrow$ &$\displaystyle \frac{2x^3}{3}-\frac{x^2}{2}$ &  \hspace{3mm} 
	$\displaystyle \left(\int\left(\frac{x}{2}-\frac{2x^2}{3}\right)\,dx =\frac{x^2}{4}-\frac{2x^3}{9} +C.\right)$ \\
	\end{tabular}
\end{center}
Hence,
\begin{align}
	\int (2x^2-x)\ln{x}\,dx	
	&= \left( \frac{2x^3}{3}-\frac{x^2}{2}\right)\ln{x} + \frac{x^2}{4}-\frac{2x^3}{9} +C. \label{eqn:2ndTable}
\end{align}
Carefully combining \eqref{eqn:SoFar} and \eqref{eqn:2ndTable} yields
\begin{align*}
	\int (3x^2-x)\ln^2{x}\,dx
	&= \left(x^3-\frac{x^2}{2}\right)\ln^2{x} + \left(\frac{x^2}{2}-\frac{2x^3}{3}\right)\ln{x} +\frac{2x^3}{9} -\frac{x^2}{4}+C_0.
\end{align*}
Note that our choice of $u$ in the second table is not the derivative of $\frac{2}{x}\ln{x}$. That is, in this example we did not iterate IBP by setting $u_j:=u_{j-1}'$ and $v_j:=\int v_{j-1}\,dx$ for $j\geq 2$, but we did apply tabular IBP twice.
\end{example} 

Other results from calculus and analysis that follow nicely from tabular IBP are explored below.

\begin{proposition}
Let $n$ be a positive integer and $x>0$. Then
\begin{align}
	\int \ln^n{x}\,dx &= x\sum_{k=0}^n(-1)^{n-k}\frac{n!}{k!}\ln^k{x}+C. \label{eqn:LogPower}
\end{align}
\end{proposition}

\begin{proof}
Let $u=\ln^n{x}$ and $dv=1\,dx$.
\begin{center}
	\begin{tabular}{c|cccl} 
	(alt.)	&(diff.)&	& (int.)	&\\
	$+/-$ 	& $u$  	& 	& $dv$	&\\ \hline \\
	$+$		& $\ln^n{x}$ 	& 	& $1$	& \\
			&		& $\searrow$	&	& \\	
	$-$		& $\displaystyle \frac{n\ln^{n-1}{x}}{x}$	&  $\rightarrow$ & $x$ &  \hspace{5mm} 
	$\displaystyle \left(-n\int\ln^{n-1}{x}\,dx\right)$ \\
	\end{tabular}
\end{center} 
Hence, $\displaystyle \int\ln^n{x}\,dx = x\ln^n{x}-n\int\ln^{n-1}{x}\,dx+C.$ Combining this result with an induction argument yields \eqref{eqn:LogPower}.
\end{proof}

Theorem \ref{thm:Taylors} provides a statement of Taylor's Formula with Integral Remainder. Its proof follows readily from tabular IBP. 

\begin{theorem}[Taylor's Formula with Integral Remainder]
\label{thm:Taylors}
If $f:\R\to\R$ has $n+1$ continuous derivatives on an interval $I$ containing $a$, then for all $x$ in $I$ we have
\begin{align*}
	f(x)	=& f(a)+f^{(1)}(a)(x-a)+\frac{f^{(2)}(a)}{2!}(x-a)^2 +\cdots+\frac{f^{(n)}(a)}{n!}(x-a)^n\\
	&+\int_a^x\frac{f^{(n+1)}(t)}{n!}(x-t)^n,
\end{align*}
where $f^{(j)}$ denotes the $j$th derivative of $f$.
\end{theorem}

The following proof is essentially identical to the one presented in \cite{Horo}.

\begin{proof}
By the fundamental theorem of calculus, we have 
\begin{align*}
	f(x)-f(a)	&=\int_a^xf^{(1)}(t)\,dt = \int_a^x-f^{(1)}(t)(-1)\,dt.
\end{align*}
Choose $u=-f^{(1)}(t)$ and, hence, $dv=-1\,dt$. In the table below, integration is performed with respect to $t$ and the first antiderivative of $-1$ is taken to be $(x-t)$, where $x$ is treated as a constant. We have
\begin{center}
	\begin{tabular}{c|ccl} 
	(alt.)	&(diff.)&	& (int.)	\\
	$+/-$& $u$  	& 	& $dv$	\\ \hline \\
	$+$	& $-f^{(1)}(t)$ 	& 	& $-1$	\\
			&		& $\searrow$	&	 \\	
	$-$	& $-f^{(2)}(t)$ 	& 	& $(x-t)$	\\
			&		& $\searrow$	&	 \\
	$+$	& $-f^{(3)}(t)$ 	& 	& $\displaystyle -\frac{(x-t)^2}{2!}$	\\
	$\vdots$	& $\vdots$&	& \hspace{5mm}$\vdots$ \\					
	$(-1)^{n-1}$	& $-f^{(n)}(t)$	&	&$\displaystyle (-1)^{n}\frac{(x-t)^{n-1}}{(n-1)!}$ \\
			&		& $\searrow$	&	 \\
	$(-1)^{n}$	& $-f^{(n+1)}(t)$	&$\rightarrow$	&$\displaystyle (-1)^{n+1}\frac{(x-t)^n}{n!}$
	\end{tabular}
\end{center}
Therefore,
\begin{align*}
	f(x)-f(a) 	=&\int_a^x-f^{(1)}(t)(-1)\,dt\\
	=& \left[-f^{(1)}(t)(x-t)-\frac{f^{(2)}(t)}{2!}(x-t)^2 -\cdots-\frac{f^{(n)}(t)}{n!}(x-t)^n \right]_a^x \\
		&+\int_a^x\frac{f^{(n+1)}(t)}{n!}(x-t)^n\,dt.
\end{align*}
The result follows immediately.
\end{proof}

\begin{remark}
Given a function $f:[0,\infty)\to\R$, its Laplace transform 
\begin{align*}
	\calL\{f(t)\}		&=\int_0^\infty e^{-st}f(t)\,dt
\end{align*}
is often determined by an application of IBP. Indeed, the integrand is a product and it is easy to both integrate and differentiate $e^{-st}$. Similarly, Laplace transform formulas, such as the one for the $n$th derivative of a function, follow in an especially nice way from tabular IBP. See \cite{Horo} for details and further results in calculus and analysis that follow readily from tabular IBP, including a proof of the Residue Theorem.
\end{remark}

%
%

The following example uses tabular IBP to derive an asymptotic expansion.

\begin{example}
For $x>0$ and any positive integer $n$,
\begin{align*}
	f(x) = \int_x^\infty t^{-1}e^{x-t}\,dt	=&~\frac{1}{x}-\frac{1}{x^2}+\frac{2!}{x^3}-\cdots + (-1)^{n-1}\frac{(n-1)!}{x^n}\\ &\hspace{3mm}+(-1)^nn!\int_x^\infty t^{-n-1}e^{x-t}\,dt.
\end{align*}
To see why this is the case, apply tabular IBP with $u=t^{-1}$ and $dv=e^{x-t}\,dt$.
\begin{center}
	\begin{tabular}{c|ccl} 
	(alt.)	&(diff.)&	& (int.)	\\
	$+/-$& $u$  	& 	& $dv$	\\ \hline \\
	$+$	& $t^{-1}$ 	& 	& $e^{x-t}$	\\
			&		& $\searrow$	&	 \\	
	$-$	& $-t^{-2}$ 	& 	& $-e^{x-t}$	\\
			&		& $\searrow$	&	 \\
	$+$	& $2t^{-3}$ 	& 	& $e^{x-t}$	\\
	$\vdots$	& $\vdots$&	& \hspace{5mm}$\vdots$ \\					
	$(-1)^{n-1}$	& $(-1)^{n-1}(n-1)!\,t^{-n}$	&	&$(-1)^{n-1}e^{x-t}$ \\
			&		& $\searrow$	&	 \\
	$(-1)^{n}$	& $(-1)^{n}n!\,t^{-n-1}$	&	$\rightarrow $&$(-1)^{n}e^{x-t}$
	\end{tabular}
\end{center}
Hence, 
\begin{align*}
	\int_x^\infty t^{-1}e^{x-t}\,dt	=&~\left[ 
	\frac{-e^{x-t}}{t}+\frac{e^{x-t}}{t^2}-\frac{2!\,e^{x-t}}{t^3}-\cdots + (-1)^{n}\frac{(n-1)!\,e^{x-t}}{t^n}\right]_x^\infty \\ 	&\hspace{3mm}+(-1)^nn!\int_x^\infty t^{-n-1}e^{x-t}\,dt. 
\end{align*}
The result follows immediately. 

With a bit more work, we have the following asymptotic expansion:
\begin{align*}
	f(x)= \int_x^\infty t^{-1}e^{x-t}\,dt	\sim&~\frac{1}{x}-\frac{1}{x^2}+\frac{2!}{x^3}-\cdots + (-1)^{n-1}\frac{(n-1)!}{x^n}+\cdots
\end{align*}
as $x \to \infty$. (Here $g(x) \sim h(x)$ as $x\to\infty$ means $g(x)/h(x)\to 1$ as $x\to\infty$.) See \cite[Example 7.2.2]{HoffMars}.
\end{example}


\begin{xca}
Evaluate $\int x^n\sin{ax}\,dx$ where $n\in\N$ and $a\neq 0$. \
\end{xca}

\begin{xca}
Evaluate $\int x^2e^x\sin{x}\,dx$.
\end{xca}

\begin{xca}
Evaluate $\int\ln{(x^2+4x+7)}\,dx$.
\end{xca}

\begin{xca}
Show that $n^b\int_0^1 (1-s)^ns^{b-1}\,ds = n!n^b/(b(b+1)\cdots (b+n))$ where $n\in\N$ and $b>0$. (This integral is used to study the Gamma Function in \cite[p.~418]{Horo}.)
\end{xca}

\vfill\eject

\end{document}